\documentclass[letterpaper,12pt]{amsart}
\usepackage{graphicx}
\usepackage{color}

\newcommand{\ka}{K\"ahler }
\usepackage[abbrev]{amsrefs}

\newtheorem{thm}{Theorem}[section]
\newtheorem{cor}[thm]{Corollary}
\newtheorem{lem}[thm]{Lemma}
\newtheorem{prop}[thm]{Proposition}
\theoremstyle{definition}
\newtheorem{definition}[thm]{Definition}
\theoremstyle{remark}
\newtheorem{rem}[thm]{Remark}
\numberwithin{equation}{section}
\newcommand{\set}[1]{\left\{#1\right\}}

\newcommand{\PE}{{\mathbb{P}E^*}}
\newcommand{\OPE}{\mathcal{O}_{\mathbb{P}E^*}(1)}
\newcommand{\ddbar}{\sqrt{-1}\,\bar{\partial}\partial}
\newcommand{\Ric}{\textrm{Ric}}
\newcommand{\Lie}[1]{\mathfrak{#1}}
\newcommand{\pa}{\partial}
\begin{document}

\title[]{Extremal Metrics On Ruled Manifolds}%

\address{University of California, Irvine,  Department of Mathematics}%

 \author{Zhiqin Lu} \address{Department of
Mathematics, University of California,
Irvine, Irvine, CA 92697, USA} \email[Zhiqin Lu]{zlu@uci.edu}

 \author{Reza Seyyedali} \address{Department of Pure
Mathematics, University of Waterloo,
Waterloo, Ontario, N2L 3G1, Canada}
\email[Reza Seyyedali]{rseyyeda@uwaterloo.ca}

\thanks{
The first author  is partially supported by the DMS-12-06748 of the National Science Foundation of USA}
 \date{September 1, 2012}



\date{November 14, 2012}
%\dedicatory{}%
%\commby{}%
% ----------------------------------------------------------------
\begin{abstract}
In this paper, we consider a compact \ka manifold with extremal \ka metric and a Mumford stable holomorphic bundle over it. We proved that, if the holomorphic vector field defining the extremal \ka metric is liftable to the bundle and if the bundle is relatively stable with respect to  the action of automorphisms of the manifold, then there exist extremal \ka metrics on the projectivization of the dual vector bundle. 
\end{abstract}
\maketitle
\tableofcontents
% ----------------------------------------------------------------
%\pagestyle{plain}
\pagenumbering{arabic}

\newtheorem*{Mthm}{Main Theorem}
\newtheorem{Thm}{Theorem}[section]
\newtheorem{Prop}[Thm]{Proposition}
\newtheorem{Lem}[Thm]{Lemma}
\newtheorem{Cor}[Thm]{Corollary}
\newtheorem{Def}[Thm]{Definition}
\newtheorem{Guess}[Thm]{Conjecture}
\newtheorem{Ex}[Thm]{Example}
\newtheorem{Rmk}{Remark}
\newtheorem{Not}{Notation}
\def\thesection{\arabic{section}}
\renewcommand{\theThm} {\thesection.\arabic{Thm}}
\baselineskip=16pt

\section{Introduction}

Let $(M,\omega)$ be a K\"ahler manifold of dimension $m$ and $L$ be an
ample line bundle over  $M$ such that $\omega \in 2\pi c_{1}(M)$. Let $\pi:E \to M$
be a holomorphic vector bundle  of rank $r\geq 2$. This gives a
holomorphic fibre bundle $\mathbb{P}E^*$ over $M$ with fibre
$\mathbb{P}^{r-1}$. We denote the tautological line bundle on $\mathbb{P}E^*$ by $\mathcal{O}_{\mathbb{P}E^*}(-1)$ and its dual bundle
by $\mathcal{O}_{\mathbb{P}E^*}(1)$.  By the Kodaira embedding theorem, for $k \gg 0$, the line bundles $ \mathcal{O}_{\mathbb{P}E^*}(1) \otimes \pi^* L^{k}$ on $\PE$ are very ample.

%Since $L$ is ample, there is an integer $k_{0} $ so that for any $k\geq k_{0}$, $\mathcal{O}_{\mathbb{P}E_{k}^*}(1)$ is very ample over $\mathbb{P}E_{k}^*$, where $E_{k}=E\otimes L^{\otimes k}$. Note that there is a canonical isomorphism $\mathbb{P}E_{k}^* \cong \mathbb{P}E^*$ and $\mathcal{O}_{\mathbb{P}E_{k}^*}(1)\cong \mathcal{O}_{\mathbb{P}E^*}(1) \otimes \pi^* L^{k}$.
In \cites{H1,H2}, Hong proved that if $E$ is  Mumford stable; $\omega$ has constant scalar curvature; and  $M$ does not admit any nontrivial  holomorphic vector fields, then $\PE$ admits cscK metric in the class of $\mathcal{O}_{\mathbb{P}E^*}(1) \otimes \pi^* L^{k}$ for $k \gg 0$. In \cite{H3}, he generalized the  result to the case that the base manifold has nontrivial automorphism group. He proved that if all Hamiltonian  holomorphic vector fields on $M$ can be lifted to holomorphic vector fields on $\PE$ and the corresponding Futaki invariants vanish, then $\PE$ admits cscK metrics in the class of $\mathcal{O}_{\mathbb{P}E^*}(1) \otimes \pi^* L^{k}$ for $k \gg 0$.  The result was  further generalized   by replacing  the liftiblity of holomorphic vector fields by a stability condition(cf. \cite{H4}). Hong considered the action of $\textup{Aut(M)}$ on the space of holomorphic structures on $E$ and showed that if $E$ is stable under this action, then there exist cscK metrics on $(\PE, \mathcal{O}_{\mathbb{P}E^*}(1) \otimes \pi^* L^{k})$ for $k \gg 0$. The  stability assumption is used to  perturb approximation  solutions to genuine cscK metrics.

In this article, we generalize Hong's result to the case that the base admits an extremal metric. Our main theorem is the following

\begin{thm}\label{thm1}

Let $(M,L)$ be a compact polarized manifold and $\omega_{\infty} \in c_{1}(L) $ be an extremal K\"ahler metric. Let  $X_{s}$ be the gradient vector field of the scalar curvature of $\omega_{\infty}$, i.e. $d S(\omega_{\infty})= \iota_{X_{s}}\omega_{\infty}$. Let $E$ be a Mumford stable holomorphic vector bundle over $M$. Suppose that the holomorphic vector field $X_{s}$ can be lifted to a holomorphic vector field on $\mathbb{P}E^*$. If $E$ is relatively stable under the action of $\textup{Aut}(M)$ in the sense of Definition \ref{def6}, then there exist extremal metrics on  $(\mathbb{P}E^*,\mathcal{O}_{\mathbb{P}E^*}(1)\otimes  \pi^*L^k)$ for $k \gg 0$.

\end{thm}

We  follow the  ideas of \cites{H4,Sz}. Let $G=\textup{Ham}(M,\omega_{\infty})$ be the group of Hamiltonian  isometries of $(M,\omega_{\infty})$ and $\Lie { g}$ be its Lie algebra. Let $G_{E}$ be the subgroup of all Hamiltonian isometries of $(M,\omega_{\infty})$ that can be lifted to  automorphisms of $\PE$. Let $\Lie{g}_{E} $ be the  Lie algebra of $G_{E}$, i.e.,  space of all Hamiltonian  holomorphic vector fields $X$ on $M$ that are liftable to  holomorphic vector fields $\tilde{X}$ on $\PE$. Fix $T\subseteq G_{E}$  a maximal torus and $K \subseteq G$  the subgroup of all elements in $G$ that commute with $T$. Let $\Lie{t}$ and $\Lie{k}$  be the Lie algebras of $T$ and $K$ respectively. We denote the space of all Hamiltonians whose gradient vector fields are in $\Lie{t}$ and $\Lie{k}$  by $\bar{\Lie{t}}$ and $\bar{\Lie{k}}$  respectively (including constant functions). Suppose that $E$ is Mumford stable. Then the Donaldson-Uhlenbeck-Yau Theorem implies that $E$ admits a Hermitian-Einstein metric $h$. The metric $h$ induces a hermitian metric $g=\hat{h}$ on $\OPE$. The restriction of the $(1,1)$-form $$\omega_{g}=i\bar{\partial}\partial \log g= i\bar{\partial}\partial \log \widehat{h}$$ on fibres are Fubini-Study metrics and therefore $\omega_{g}|_{{\rm Fiber}}$ is non-degenerate. Hence for $k \gg 0$, the  $(1,1)$-forms $\omega_{k}=\omega_{g}+k\omega_{\infty}$ define  K\"ahler metrics. Finding extremal metrics on $(\mathbb{P}E^*,\mathcal{O}_{\mathbb{P}E^*}(1)\otimes  L^k)$ is equivalent to finding $\phi \in \mathcal C^\infty(\PE)^T$ and $f \in \bar{\Lie{t}}$ such that
\begin{equation}\label{eq1}  S(\omega_{k}+ \ddbar \phi)+\frac{1}{2}\langle \nabla f, \nabla \phi \rangle=f,\end{equation} where $\nabla$ and $\langle\,\, , \,\,\rangle$ are taken with respect to $\omega_{k}$, and  $\mathcal C^\infty(\PE)^T$ is the space of smooth functions on $\PE$ that are invariant under the action of $T$. 
To see that $\omega_k+\ddbar\phi$ is an extremal metric, we assume that
\[
df=\iota(X)\omega_k
\]
for some holomorphic vector field $X$. We write $X=X_1+\bar X_1$ for holomorphic $(1,0)$ vector field $X_1$. Then a straightforward computation shows that
\[
\bar\pa S=\iota(X_1)(\omega_k+\ddbar\phi).
\]

Our  strategy is to replace  equation \eqref{eq1} with  the one  that is easier to solve and is relating  it to a finite dimensional GIT problem (cf. \cites{H4,Sz}). The first step is to find $\phi \in \mathcal C^\infty(\PE)^T$ and $b \in \bar{\Lie{k}}$ such that
\begin{equation} \label{eq2} S(\omega_{k}+ \ddbar \phi)+\frac 12\langle \nabla l_{k}(b), \nabla \phi \rangle=l_{k}(b),\end{equation}
where $\nabla$ and $\langle\,\, , \,\,\rangle$ are taken with respect to $\omega_{k}$ and $l_{k}(b)$ is a lift of $b$ to $\PE $ defined in Definition \ref{def4}. Note that if $b \in \bar{\Lie{t}}$, then $\omega_{k}+\ddbar \phi$ is an extremal metric. Allowing  $b$ to be in a slightly larger space makes it easier to solve the equation.  In order to solve equation \eqref{eq2}, 
we first construct K\"ahler forms $\omega_{k,p} $ in the class of $\omega_{k}$  for any positive integer $p$ and $k \gg 0$ as approximation solutions.
We then  apply contraction mapping theorem.

\begin{thm}\label{thm2}
Let $p \geq 6$ be an integer. Suppose that $X_{s} \in \Lie{t}$. Then for any $k \gg 0$, we can find $\phi \in \mathcal C^\infty(\PE)^T$ and $b \in \bar{\Lie{k}}$ such that
$$  S(\omega_{k,p}+ \ddbar \phi)+\frac{1}{2}\langle \nabla l_{k,p}(b), \nabla \phi \rangle=l_{k,p}(b).$$ Here $\nabla$ and $\langle\,\, , \,\,\rangle$ are taken with respect to $\omega_{k,p}$. Moreover $b $ has the following expansion:
$$b=r(r-1)+k^{-1}S(\omega)-k^{-2}\pi_{N}(\Sigma_{E})+O(k^{-3}),$$ where $\pi_{N}:\mathcal C^\infty(M) \to \ker(\mathcal{D}^*\mathcal{D})$ and
\begin{align*} &\Sigma_{E}=\frac{2}{r} \Lambda^2 (\textup{Ric}(\omega)\wedge \textup{tr}(iF_{h}))-\frac{2}{r(r+1)} \Lambda^2 (\textup{tr}(iF_{h})\wedge \textup{tr}(iF_{h}))\\&
+\frac{2}{r+1}\Lambda^2 \textup{tr}(iF_{h}\wedge iF_{h})-\mu S(\omega)+\frac{\textup{tr}(u_{X_{s}})}{r}.
\end{align*}
See Proposition \ref{prop3} and \ref{def5} for the definition of $u_{X_{s}}$ and $l_{k,p}$ respectively.

\end{thm}
The metric $\omega_{k,p}+ \ddbar \phi$ would have been an extremal metric if $b $ were in  $ \bar{\Lie{t}}$. If not, we perturb the holomorphic structure on $E$ so that the Hamiltonian  $b$ lies in $\bar{\Lie{t}}$ after the perturbation. This can be done by applying implicit function theorem using the  stability assumption. 
In a recent paper, Br\"onnle~\cite{B}, using the similar method,   proved that if the base is cscK without holomorphic vector fields and the bundle is a direct sum of stable bundles with different slopes, then the projectivization admits extremal metrics.

The outline of the paper is as follows: In section $2$, we go over some basic facts and definitions. In section $3$, we compute an expansion for the scalar curvature of the metrics $\omega_{k}$. Section $4$ is devoted to the construction of K\"ahler metrics $\omega_{k,p}$. In Section $5$, we prove Theorem  \ref{thm2}. In the last section, we adopt Hong's moment map setting to our situation and prove the main theorem.

\section{Preliminaries}
Let $V$ be a hermitian vector space of dimension $r$. The
projective space $\mathbb{P}V^*$ can be identified with the space
of hyperplanes in $V$ via $f\in V^* \rightarrow
{{\rm ker}(f)}=V_{f}\subseteq V$.  There is a natural isomorphism between $V$ and
$H^{0}(\mathbb{P}V^*,\mathcal{O}_{\mathbb{P}V^*}(1))$ which sends
$v\in V$ to $\hat{v}\in
H^{0}(\mathbb{P}V^*,\mathcal{O}_{\mathbb{P}V^*}(1)) $ such that
for any $f \in V^*, \hat{v}(f)=f(v) $.

 \begin{definition}
For any hermitian inner product  $h$ on $V$, we use $\langle\,\cdot\,,\,\cdot\,\rangle_h$ to denote the hermitian inner product induced by $h$ and we use $\|\,\cdot\,\|_h$ to denote the norm with respect to $h$ on both $V$ and $V^*$.
The hermitian inner product $h$ induces a
hermitian metric on $\mathcal{O}_{\mathbb{P}V^*}(1)$, which can be explicitly represented as follows: for $v,w \in V$ and $f
\in V^*$ we define
\begin{equation}\label{12}
\langle\hat{v},\hat{w}\rangle_{\hat h}=
\frac{f(v)\overline{f(w)}}{\|f\|_h^{2}}.
\end{equation}
We denote the induced metric on
$\mathcal{O}_{\mathbb{P}V^*}(1)$ by $\widehat{h}$. \end{definition}

The following is a straightforward computation.
\begin{prop}
 For any $v,w \in V$ we have
$$\langle v,w\rangle_{h}=C_{r}^{-1}\int_{\mathbb{P}V^*} \langle\hat{v},\hat{w}\rangle_{\widehat{h}}
\frac{\omega_{\textrm{FS}}^{r-1}}{(r-1)!}$$ where $C_{r}$ is a
constant defined by
\begin{equation}\label{eq3}C_{r}=\int_{\mathbb{C}^{r-1}}
\frac{(\sqrt{-1})^{r-1}d\xi \wedge d\overline{\xi}}{(1+\sum_{j=1}^{r-1}
|\xi_{j}|^{2})^{r+1}}=\frac{(2\pi)^{r-1}}{r!},\end{equation} and $(\sqrt{-1})^{r-1}d\xi \wedge
d\overline{\xi}= (\sqrt{-1}\,d\xi_{1} \wedge
d\overline{\xi}_{1})\wedge \dots \wedge (\sqrt{-1}\,d\xi_{r-1}
\wedge d\overline{\xi}_{r-1}).$
\end{prop}

\begin{definition}\label{def0}
For any $v \in V$ and any hermitian inner product  $h$ on $V$, we define an endomorphism $\lambda(h)=\lambda(v,h)$ of $V$ by
$$\lambda(v,h)=\frac{1}{\|v\|_{h}^2} v \otimes v^{*_{h}},$$
where $v^{*_{h}}(\cdot)= h(\cdot,v)$ is the dual  element of $v$ with respect to the inner product $h$.
\end{definition}

The above  settings can be made  into the  following family version. Let $(M,\omega)$ be a K\"ahler manifold of dimension $m$ and $E$
be a holomorphic vector bundle on $M$ of rank $r\ge 2$. Let $L$ be an
ample line bundle on $M$ endowed with a hermitian metric $\sigma$
so that $i\bar{\partial}\partial \log \sigma=\omega$.
The configuration $(M,\omega, L, \sigma)$ is called a polarized K\"ahler manifold.
Let $\PE$  be the projectivization of the dual bundle $E^*$ of $E$. A hermitian metric $h$ on $E$ induces a hermitian metric $\widehat{h}$
on the line bundle $\mathcal O_{\PE}(1)$ by~\eqref{12}.

Let $\omega_g$ be   the  $(1,1)$-form  on $ \PE$ defined by
$$\omega_{g}= i\bar{\partial}\partial \log \widehat{h}.$$
Let  $\pi:\mathbb{P}E^*\rightarrow M$ be the projection map.
Define  the smooth functions $f_{1}, \dots f_{m} \in \mathcal  C^\infty(\PE)$ by
\begin{equation}\label{eq4}\frac{\omega_{g}^{r-1+j}}{(r-1+j)!} \wedge \frac{\pi^*\omega^{m-j}}{(m-j)!}=f_{j}\frac{\omega_{g}^{r-1}}{(r-1)!}\wedge \frac{\pi^*\omega^m}{m!}.\end{equation}
Alternatively, $f_j$'s can be generated by the following equation
\begin{equation}\label{eq-5}
\omega_k^{m+r-1}=\frac{(m+r-1)!}{m!(r-1)!}\sum_{j=1}^m  k^{m-j} f_j\omega_g^{r-1}\wedge\pi^*\omega^m,
\end{equation}
where
\[
\omega_k=\omega_g+k\pi^*\omega.
\]

\begin{definition}\label{def1}
Let $X$ be a compact \ka  manifold with the K\"ahler metric $\omega$. Assume that the complex dimension of $X$ is $N$.
For any $(j,j)$-form $\alpha$ on $X$, we define the contraction
$\Lambda_{\omega}^j\alpha$ of $\alpha$ with respect to the
K\"ahler form $\omega$ by $$\frac{N!}{j!(N-j)!} \alpha \wedge
\omega^{N-j}=(\Lambda_{\omega}^j\alpha )\,\,\omega^N.$$
In particular, we define
\[
\Lambda_\omega\alpha=\Lambda_\omega^1\alpha.
\]

\end{definition}

\begin{definition}
We define the vertical subbundle $V$ and the horizontal subbundle $H$ of the holomorphic tangent bundle $T(\PE)$ of $ \PE$ as follows: let  $u \in \PE $ and let $\pi(u)=x$.
\begin{align*}
&V_{u}=T_{u}(\mathbb{P}E^*_x);\\
&H_{u}=\set {\xi \in T_{u}({\PE}) \mid \omega_{g}(\xi, w)=0, \forall w \in V_{u} }.
\end{align*}
\end{definition}

Since the restriction of $\omega_{g}$ to the fibre is the Fubini-Study metric of $\mathbb{P}E^*_x$, $\omega_{g}|_{\textrm{Fibre}}$ is non-degenerate. As a result, $H$ is indeed a vector bundle of rank $m$, and we have the following (holomorphic bundle) decomposition  $$T(\PE) =H \oplus V.$$  
By the dimension consideration, we have $H^*=\pi^*(T^*M)$, where $H^*$ is the dual bundle of $H$. Let $V^*$ be the dual bundle of $V$. Then we have
\begin{equation}\label{1}
T^*(\PE)=V^*\oplus \pi^*(T^*M).
\end{equation}

Let $\bigwedge (T^*(\PE))$ be the bundle of differential forms of $\PE$. Write
$$\bigwedge (T^*(\PE))= C_{H} \oplus C_{V} \oplus C_{m},$$
where $C_{H}, C_{V}$ and $C_{m}$ are the bundles of horizontal, vertical, and mixed forms, respectively.
Note that $C_{H}= \pi^* (\bigwedge (T^*M)).$ For any differential form $\alpha$ on $\PE$, we write  $\alpha=\alpha_{H}+\alpha_{V}+\alpha_{m}$, where $\alpha_{H},\alpha_{V},\alpha_{m}$ are  the
horizontal, vertical, and mixed components of $\alpha$ respectively.

Using the above notation, we have
\begin{lem} There is no mixed component of $\omega_g$.
\end{lem}

\begin{proof} This follows from ~\eqref{1}.

\end{proof}

If we write
\[
\omega_g=(\omega_g)_H+(\omega_g)_V
\]
as its horizontal and vertical parts. Then~\eqref{eq4} can be written as
\begin{equation}\label{3}
f_j\,\pi^*(\omega^m)=\frac{m!}{j!(m-j)!}(\omega_g)_H^j\wedge\pi^*\omega^{m-j}.
\end{equation}

Let $F_h\in {\bigwedge}^{1,1}({\rm Hom}\, (E,E))$ be the curvature tensor of $h$
\[
F_h=\bar\pa\,(\pa h\cdot h^{-1}).
\]
From~\eqref{3}, we can prove the following

\begin{lem}\label{lem22}
For any $v \in E^*$, we have $$f_{j}([v])= \Lambda^j _\omega\Big( {\sqrt{-1}}\,\textup{Tr}\big(   \lambda(v, h)  F_{h}   \big)\Big)^j,$$  where $[v] \in \PE $ is the class of $v$ in $\PE$.
\end{lem}

\begin{proof}
Let
\begin{equation}\label{beta}
\beta={\sqrt{-1}}\,\textup{tr}\big(   \lambda(v, h)  F_{h}   \big)=\sqrt{-1}\|v\|_h^{-2} \langle F_h(v),v\rangle_h.
\end{equation}
Let $x=\pi(u)$.  We assume that at $x$, $\{e_1,\cdots, e_r\}$ is a {\it normal} frame. That is, under this frame
\[
h_{i\bar j}(x)=\delta_{ij},\quad dh_{i\bar j}(x)=0.
\]
Since there are no connection terms, by  a straightforward computation, we obtain\footnote{\label{footnote-1}Strictly speaking, $\beta$ is a section of the sheaf $\mathcal C^\infty(\PE)\otimes \mathcal A^{1,1}(M)$. So the $\pi^*$ operation is only acting on the second component.}
\begin{equation}\label{gamma}
\omega_g= \pi^*(\beta)+\left.\omega_{g}\right|_{\PE_x}.
\end{equation}
Therefore,
\begin{equation}\label{2}
(\omega_g)_H=\pi^*\beta.
\end{equation}
The lemma follows from Definition~\ref{def1}.

\end{proof}

Let $\alpha$ be a $(1,1)$-form on $\PE$. Define $\tilde \Lambda_{\omega_g}\alpha_V$ by
\[
\tilde \Lambda_{\omega_g}\alpha_V\wedge((\omega_g)_V)^{r-1}
=(r-1)\alpha_V\wedge
((\omega_g)_V)^{r-2}.
\]
Therefore, we have
\[
(\tilde \Lambda_{\omega_g}\alpha_V)\,{\omega_g^{m+r-j-1}}\wedge\pi^*\omega^j=(m+r-j-1)\alpha_V\wedge\omega_g^{m+r-j-2}\wedge\pi^*\omega^j
\]
for $j\geq 0$.

\begin{definition}\label{def2}
For any smooth function  $f \in \mathcal C^{\infty}(\PE)$, define  the operators $\Delta_V, \Delta_H$ and $\tilde \Delta_H$  (and call them the Laplacians) by the following equations
\begin{align*}
&(r-1) \ddbar f \wedge \omega_{g}^{r-2} \wedge \pi^*\omega^m= \Delta_{V}f \omega_{g}^{r-1} \wedge \pi^*\omega^m,\\
&m \ddbar f \wedge  \omega_{g}^{r-1} \wedge \pi^*\omega^{m-1}= \Delta_{H}f \,\omega_{g}^{r-1} \wedge \pi^*\omega^m,\\
&\tilde{\Delta}_{H}f= \Delta_{H}f-f_{1}\Delta_{V}f.
\end{align*}

\end{definition}

\begin{Rmk}

The Laplacians $\Delta_{H}$ and $\Delta_{V}$ are the same as ones defined in \cite{H1}.

\end{Rmk}

%The following proposition is proved in \cite{H1}.

%\begin{prop}
%Suppose that the hermitian metric $h$ on $E$ is Hermitian-Einstein, then$$\Delta_{V} \circ \Delta_{H}=\Delta_{H}\circ \Delta_{V}.$$
%\end{prop}

\begin{definition}\label{def3}
For any $x\in M$, we define $W_{x}$ as the space of all eigenfunctions of the  Laplacian (on functions) on $\mathbb P E_x$ (with respect to the metric $\omega_g|_{\mathbb P E_x}$)
associated to the first nonzero eigenvalue. Define the vector bundle $W$ whose fibers are $W_x$(c.f. \cite{H2}).

\end{definition}

Let $\textup{End}_{0}(E_{x})$ be the space of traceless endomorphisms of $E_x$ for any $x\in M$.
The first nonzero eigenvalue of the Laplacian  is $r$. As is well-known, $$\Phi \in \textup{End}_{0}(E_{x}) \to \textup{Tr}\big(\lambda(h)\Phi\big) \in W_{x}$$ is a $1$-$1$ correspondence.  Define  $\textup{End}_{0}(E)$ to be  the smooth vector bundle whose fibers are $\textup{End}_0(E_x)$ for any $x\in M$. Thus we have $W=\textup{End}_0(E)$.

%For any function $f \in \mathcal C^{\infty}(\PE)$, we define $\tilde{f} \in \mathcal C^{\infty }(M)$ as follows:

%$$ \tilde{f}(x)= C_{r}^{-1} \int_{\mathbb P E_x} f \frac{\omega_{g}^{r-1}}{(r-1)!}.$$Therefore we obtain the following decomposition of the space of smooth functions on $\PE$. $$\mathcal C^{\infty}(\PE)=\mathcal C^{\infty}(M) \oplus \set{f \in \mathcal C^{\infty}(\PE)\mid  \tilde{f}=0}.$$ Furthermore, we can decompose
%\[ \set{f \in \mathcal C^{\infty}(\PE)\mid \tilde{f}=0}\]  as the space of sections of $W$ and the perpendicular part. Thus, we have  $$\mathcal C^{\infty}(\PE)=\mathcal C^{\infty}(M) \oplus \Gamma(M,W) \oplus  C^{\bot},$$ where
 %\[  C^{\bot}:=\{f \in C^{\infty}(\PE)\mid  \tilde{f}=0,\,\,\, \int_{\PE_{x}} f\phi \frac{\omega_{g}^{r-1}}{(r-1)!}=0 \,\,\,\,\forall x\in M,\, \phi \in   \Gamma(M,W)  \}.\]

%\begin{lem}

%For any smooth $\Phi \in \Gamma(M,E)$, we have $$\Delta_{H} tr\big(\lambda(h)\Phi\big)= \Lambda tr\big(\lambda(h)\ddbar\Phi\big).$$

%\end{lem}

%The proposition implies that $\Delta_{H}$ preserves sections of $W$.

%\begin{prop}

 %The operator $\Delta_{H}: \Gamma(M,W) \to \Gamma(M,W) $ is invertible if and only if $E$ is simple.

%\end{prop}

\section{Scalar curvature}

The goal of this section is to find the  asymptotic expansion for the scalar curvature of the K\"ahler form $\omega_{k}=\omega_{g}+k \pi^*\omega$. The main result of this section is

\begin{thm}\label{thm3}
Let $\omega$ be a K\"ahler metric on $M$ and $h$ be a hermitian metric on $E$.
Let
\[
\omega_k=\omega_g+k\pi^*(\omega),
\]
where $k$ is a large positive integer.
Then we have the following expansion of the scalar curvature $\textup{Scal}(\omega_{k})$ of $\omega_k$
  \begin{align*}&\textup{Scal}(\omega_{k})= r(r-1)+k^{-1}(\pi^*S(\omega)+2r\,\Lambda_\omega(\textup{Tr}(\lambda(h)F_h^\circ)))\\&+k^{-2}\Big(2
\Lambda_\omega^2((\pi^*(\Ric(\omega)-\textup{Tr}(iF_h))\wedge\omega_g)_H)-f_1(\pi^*(S(\omega)-\Lambda_\omega(\textup{Tr}(iF_h))))\\
&+
\Delta_V(f_2-\frac 12 f_1^2)+\tilde\Delta_H f_1-rf_{1}^2+2rf_{2}\Big)+O(k^{-3}),\end{align*}
 where $S(\omega)$ is the scalar curvature of $\omega$ and $F^{\circ}_{h}=F_{h}-\frac{1}{r}tr(F_{h})$ is the trace-less part of the curvature tensor  of $h$. (For the definition of $f_{1}, \dots f_{m}$, $\lambda(h)$, $\tilde{\Delta}_{H}$, $\Delta_{V}$, $\Lambda_\omega $ and $\Lambda_\omega^2$,  see \eqref{eq4}, Definition \ref{def0}, Definition \ref{def1} and Definition \ref{def2}).

\end{thm}

 Let $\alpha=\pi^{*}\alpha_1$ be a horizontal form of $\PE$ (see footnote {\ref{footnote-1}}). Then we define
 \[
 \Lambda_\omega\alpha=\pi^*(\Lambda_\omega\alpha_1).
 \]

First we prove the following purely algebraic  lemmas.

\begin{lem}\label{lem1}

Let $\alpha $ be a $(1,1)$-form on $\PE$. Then
$$\Lambda_{\omega_{k}}\alpha= \tilde\Lambda_{\omega_{g}} \alpha_{V}+k^{-1}\Lambda_{\omega} \alpha_{H}+k^{-2}\big( 2\Lambda_{\omega}^2 ( \alpha \wedge \omega_{g})_{H}-(\Lambda_{\omega} \alpha_{H})f_{1}  \big)+O(k^{-3}). $$
In particular if $\alpha \in \bigwedge^{1,1}(M)$, then
$$\Lambda_{\omega_{k}}\alpha=k^{-1}\pi^*(\Lambda_\omega\alpha)+ k^{-2}\Big (  2\Lambda^2_\omega(\alpha \wedge \omega_{g})_{H}-f_{1}\,\pi^*(\Lambda_\omega\alpha)  \Big )+O(k^{-3}).$$
\end{lem}

\begin{proof}

By definition, we have $$(\Lambda_{\omega_{k}}\alpha) \omega_{k}^{m+r-1}= (m+r-1)\alpha \wedge \omega_{k}^{m+r-2}.$$ We define $g_{j}=g_{j}(\alpha) $ by the equation

$$\frac{1}{(m+r-2)!}\,\alpha \wedge \omega_k^{m+r-2}=\frac{1}{(r-1)!m!}\, k^m (\sum _{j=0}^m k^{-j}g_{j})\, \omega_{g}^{r-1}\wedge \omega^m.$$

 Let
\[
\alpha=\alpha_V+\alpha_H+\alpha_m
\]
be the decomposition of $\alpha$ into its vertical, horizontal, and mixed components.
Then we have
\begin{align*}
&\frac{(r-1)!\,m!}{(m+r-2)!}\,\alpha\wedge\omega_k^{m+r-2}\\
&
=(\tilde\Lambda_{\omega_g}\alpha_V)((\omega_g)_H+k\pi^*\omega)^m\wedge((\omega_g)_V)^{r-1}
\\&+m\alpha_H\wedge ((\omega_g)_H+k\pi^*\omega)^{m-1}\wedge((\omega_g)_V)^{r-1}\\
&
=\sum_j k^{m-j} g_j ((\omega_g)_V)^{r-1}\wedge\pi^*\omega^m.
\end{align*}

 Simple calculation shows that
  \begin{align*}
 &g_{0}=\tilde\Lambda_{\omega_{g}} \alpha_{V};\\
 &
g_{1}= \Lambda_{\omega} \alpha_{H}+(\tilde\Lambda_{\omega_{g}} \alpha_{V})f_{1};\\
&g_{2}= 2\Lambda_{\omega}^2 ( \alpha \wedge \omega_{g})_{H}+(\tilde\Lambda_{\omega_{g}} \alpha_{V})f_{2}.
\end{align*}

By~\eqref{eq-5}, the above equation implies
\begin{align*}
&\Lambda_{\omega_{k}}\alpha= \frac{\sum k^{-j}g_{j}}{\sum k^{-j}f_{j}}\\&= g_{0}+k^{-1}(g_{1}-g_{0}f_{1})+k^{-2}(g_{2}-g_{1}f_{1}-g_{0}f_{2}+g_{0}f_{1}^2)+O(k^{-3}).\end{align*}

The lemma is proved.
%where $g_{j}= \Lambda^{j+1}(\alpha \wedge \omega_{g})_{H}^j$.
%Therefore $$\Lambda_{\omega_{k}}\alpha=k^{-1}\Lambda\alpha+ k^{-2}\Big ( -f_{1}\Lambda\alpha  + \Lambda^2(\alpha \wedge \omega_{g})_{H}\Big )+O(k^{-3}).$$

\end{proof}

Let $\Delta_k$ be the Laplacian with respect to the metric $\omega_k$. That is,
\[
\Delta_{k}f=\Lambda_{\omega_{k}} (\ddbar f)
\]
for smooth functions $f$ on $\PE$. Then we have the following asymptotics:

\begin{lem}\label{lem2}
For any $f \in \mathcal C^{\infty}(\PE)$, we have
$$\Delta_{k}f= \Delta_{V}f+k^{-1}\tilde{\Delta}_{H}f+k^{-2} \Big ( -f_{1}\tilde{\Delta}_{H}f  + 2\Lambda_\omega^2(\ddbar f \wedge \omega_{g})_{H}\Big )+O(k^{-3})$$
as $k\to\infty$.

\end{lem}

\begin{proof}
Let
$\alpha=\ddbar f$. Then we have
\[
\Delta_k f=\Lambda_{\omega_k}\alpha.
\]
By Lemma~\ref{lem1}, we have
\[
\Delta_k f=\tilde\Lambda_{\omega_{g}} \alpha_{V}+k^{-1}\Lambda_{\omega} \alpha_{H}+k^{-2}\big(2 \Lambda_{\omega}^2 ( \alpha \wedge \omega_{g})_{H}-(\Lambda_{\omega} \alpha_{H})f_{1}  \big)+O(k^{-3}).
\]
By Definition~\ref{def2}, we have
\[
(\tilde\Lambda_{\omega_g}\alpha_V)\omega_g^{r-1}\wedge\pi^*\omega^m=(r-1)\alpha\wedge\omega_g^{r-2}\wedge\pi^*\omega^m
=(\Delta_V f)\omega_g^{r-1}\wedge\pi^*\omega^m.
\]
Thus
\[
\tilde\Lambda_{\omega_g}\alpha_V=\Delta_V f.
\]
Similarly, we have
\[
\tilde\Lambda_{\omega_g}\alpha_V\, f_1+\Lambda_\omega\alpha_H=\Delta_H f.
\]
Thus we have
\[
\Lambda_\omega\alpha_H=\tilde\Delta_H f.
\]
The lemma is proved.

\end{proof}

\begin{proof}[Proof of Theorem \ref{thm3}]
we have the following exact sequence of holomorphic vector bundles on $\PE$.
$$0 \to V \to T\PE \to \pi^*TM \to 0. $$ The hermitian metric $h$ on $E$ induces a  Fubini-Study metric $h_{FS}$ on $V$. The positive $(1,1)$-forms $\omega_{k}$ and $(\omega_{g})_{H}+k\pi^*\omega$ induce hermitian metrics on vector bundles $T\PE$ and $\pi^*TM$ respectively.  As  holomorphic  hermitian vector bundles, the above exact sequence splits in the smooth category:
$$(T\PE, \omega_{k})=(V, h_{FS}) \bigoplus (\pi^* TM, (\omega_{g})_{H}+k\pi^*\omega)$$ and in addition, we have  $$\Ric(\omega_{k})= \textup{Tr}(i F_{h_{FS}})+\Ric((\omega_{g})_{H}+k\pi^*\omega).$$
On the other hand, we have the following Euler sequence of holomorphic vector bundles on $\PE$.
$$0 \to \mathbb{C} \to \pi^*E^* \otimes \OPE \to V \to 0. $$ This gives the following isometric isomorphism of holomorphic line bundles on $\PE$. $$(\det(V),\det(h_{FS})) \cong (\det(\pi^*E \otimes \OPE), \det(\pi^*h \otimes \hat{h})). $$
Therefore (cf.~\eqref{gamma}), $\textup{Tr}(i F_{h_{FS}})=r \omega_{g}-\pi^*\textup{Tr}(iF_{h})$, and we have $$\Ric(\omega_{k})= r\omega_{g}+\Ric((\omega_{g})_{H}+k\pi^*\omega)-\pi^*\textup{Tr}(iF_{h}).$$
On the other hand, by~\eqref{3}, we have
$$k^{-m}\big(  (\omega_{g})_{H} +k\pi^*\omega\big)^m= (1+k^{-1}f_{1}+ \cdots +k^{-m }f_{m}) \,\pi^*\omega^m.$$
As a result,
\begin{align*}
&\Ric((\omega_{g})_{H}+k\pi^*\omega)=  \ddbar \log ((\omega_{g})_{H} +k\pi^*\omega\big)^m\\
&=\ddbar \log (\sum_{j=0}^m
k^{-j}f_{j})+\pi^*(\Ric(\omega)).
\end{align*}
Consequently, \begin{equation}\label{eq5}\Ric\,(\omega_{k})= r\omega_{g}-\pi^*\textup{Tr}(iF_{h})+\pi^*(\Ric(\omega))+ \ddbar \log (\sum_{j=0}^mk^{-j}f_{j}).\end{equation} Taking trace of \eqref{eq5} with respect to $\omega_{k}$, we get
 \[
\textup{Scal}(\omega_k)=  \Lambda_{\omega_k}\alpha+\Delta_k\log (\sum_{j=0}^mk^{-j}f_{j}),
 \]
 where
 \[
 \alpha=\pi^*(\Ric(\omega)-\textup{Tr}(iF_h))+r\omega_g.
 \]
  Let
 \[
 b=\pi^*(S(\omega)-\Lambda_\omega(\textup{Tr}(iF_h))).
\]
 Using  Lemma \ref{lem1}, we get
 \begin{align*}
 &
 \Lambda_{\omega_k}\alpha=r(r-1)+k^{-1}(b+ rf_1)\\
 &
 +k^{-2}(2\Lambda_\omega^2(\pi^*(\Ric(\omega)-\textup{Tr}(iF_h))\wedge\omega_g)_H-f_1b-rf_1^2+2rf_2)+O(k^{-3}).
 \end{align*}
 By  Lemma \ref{lem2}, we have
\[
\Delta_k\log (\sum_{j=0}^mk^{-j}f_{j})=k^{-1}(\Delta_V f_1)+k^{-2}(\Delta_V(f_2-\frac 12 f_1^2)+\tilde\Delta_H f_1) +O(k^{-3}).
\]
Therefore, we have
\begin{align*}
&\textup{Scal}(\omega_{k})= r(r-1)+k^{-1}(b+ rf_1+\Delta_V f_1)\\
&+k^{-2}(2
\Lambda_\omega^2(\pi^*(\Ric(\omega)-\textup{Tr}(iF_h))\wedge\omega_g)_H-f_1b\\
&+
\Delta_V(f_2-\frac 12 f_1^2)+\tilde\Delta_H f_1-rf_{1}^2+2rf_{2})+O(k^{-3}).
\end{align*}
On the other hand, by the discussion at the end of the last section, we have $\Delta_{V}f_{1}=rf_{1}-\Lambda_\omega \textup{Tr}(iF_{h}).$ This concludes the proof.

\end{proof}

An easy computation shows the following

\begin{cor}[c.f. \cite{H4}]\label{cor1}
Suppose that $h$ is a Hermitian-Einstein metric on $E$ with respect to $\omega$, i.e. $\Lambda_{\omega} (iF_{h})=\mu I_{E}$, where
$\mu$ is the $\omega-$slope of the bundle $E$. Then for any $x\in M$, we have
 \begin{align*}
 &\frac{1}{(2\pi)^{r-1}}\int_{\mathbb PE^*_x} \textup{Scal}(\omega_{k})\, \omega_{g}^{r-1}= C(k)+k^{-1}S(\omega)+k^{-2}\Big(\frac{2}{r} \Lambda^2_\omega (\textup{Ric}(\omega)\wedge \textup{Tr}(iF_{h}))-\\&\frac{2}{r(r+1)} \Lambda^2_\omega (\textup{Tr}(iF_{h})\wedge \textup{Tr}(iF_{h}))+\frac{2}{r+1}\Lambda^2_\omega \textup{Tr}(iF_{h}\wedge iF_{h})-\mu S(\omega)\Big)+O(k^{-3}),\end{align*} where $C(k)$ is a constant depends on $k$.

\end{cor}

\section{Construction of approximate solutions}

In this section, we first compute the linearization of the scalar curvature operator at  the K\"ahler metrics $\omega_{k}$.

\begin{prop}{\cite{F}}\label{prop0}
Let $(Y,\omega)$ be a K\"ahler manifold of dimension $n$. Then the linearization of the scalar curvature operator at the K\"ahler metric $\omega$ is given by the following formula.
$$L(\phi)=(\Delta^2-S(\omega) \Delta)\phi+n(n-1)\frac{\ddbar \phi \wedge \textup{\Ric}(\omega) \wedge \omega^{n-2} }{\omega^n},$$
where $\phi$ is a smooth function on $Y$.

\end{prop}

Applying the above proposition to $(\PE, \omega_{k})$, we obtain the following.

\begin{prop}\label{prop1}
Let $L_{k}$ be the linearization of the  scalar curvature operator at K\"ahler metrics $\omega_{k}$. Then we have the following

$$L_{k}=\Delta_{V}(\Delta_{V}-r)+O(k^{-1}).$$

\end{prop}

\begin{proof} By~\eqref{eq5}, we have
\[
\ddbar \phi \wedge \textup{\Ric}(\omega_k) \wedge \omega_k^{n-2}=C_{r-3+n}^n\ddbar \phi \wedge\omega_g\wedge\pi^*\omega^n+O(k^{n-1}).
\]
Since $\textup{Scal}(\omega_k)=r(r-1)+O(k^{-1})$ by Theorem~\ref{thm3}, we have
\[
(n+r-1)(n+r-2)\frac{\sqrt{-1}\bar\pa\pa\phi\wedge\textup{Ric}(\omega_k)\wedge\omega_k^{n+r-3}}
{\omega_k^{n+r-1}}=r(r-2)\Delta_V+O(k^{-1}).
\]
The result follows from Proposition~\ref{prop0}.

\end{proof}

%\begin{prop}\label{prop2}

%Suppose  $F \in C^{\bot}$, i.e  $f \in C^{\infty}(\PE)$ such that $$\int_{\textup{Fibre}} F \omega_{g}^{r-1}=0,\quad  \int_{\textup{Fibre}} F\phi \omega_{g}^{r-1}=0$$ for any $\phi \in \Gamma (M,W)$. Then there exists a unique function $G \in C^{\bot} $ such that $$\Delta_{V}(\Delta_{V}-r)G=F.$$

%\end{prop}

%\begin{proof}

%The operator $\Delta_{V}(\Delta_{V}-r)$ is an elliptic operator of order $4$. By definition, the kernel of $\Delta_{V}(\Delta_{V}-r)$ is equal to $\mathcal C^{\infty}(M)\oplus \Gamma(M,W)$. Therefore elliptic theory concludes the proposition.

%\end{proof}

We make the following definition of a holomorphic vector field. Let $X$ be a $(1,0)$-vector field such that $\bar\pa X=0$. Then $X+\bar X$ is a real vector field and it is called a holomorphic vector field. A holomorphic vector field generates a one-parameter group of holomorphic automorphisms.

Let $\omega_{\infty}$ be an extremal metric on $M$ and $X_{s}$ be the holomorphic vector field such that  $d S(\omega_{\infty})= \iota_{X_{s}}\omega_{\infty}$.

Let $G=\textup{Ham}(M,\omega_{\infty})$ be the group of Hamiltonian isometries of $(M,\omega_{\infty})$ and $\Lie { g}$ be its Lie algebra. Let $G_{E}$ be the subgroup of all Hamiltonian isometries of $(M,\omega)$ that can be lifted to  automorphisms of $\PE$ and let $\Lie{g}_{E} $ be its  Lie algebra.
$\Lie{g}_{E} $ is the space of  holomorphic vector fields $X$ on $M$ such that
\begin{enumerate}
\item there exist holomorphic vector fields $\tilde X$ of $\PE$ such that $\pi_*\tilde X=X$;
\item there exist real valued   functions $f$ such that $d f=\iota_{X}\omega_\infty$.
\end{enumerate}

Let $\mathfrak h$ be a Lie sub algebra of $\mathfrak g$. We denote the space of all Hamiltonians  (including constant functions) whose gradient vector fields are in $\Lie{h}$ by $\bar{\Lie{h}}$.
Fix $T\subseteq G_{E}$  a maximal torus and $K \subseteq G$  the subgroup of all elements in $G$ that commute with $T$. Let $\Lie{t}$ and $\Lie{k}$  be the Lie algebras of $T$ and $K$ respectively. Suppose that $b \in \bar{\Lie{k}}$. By definition, there exists a holomorphic vector field $X$ on $M$ such that $d b=\iota_{X}\omega_{\infty}$. If we further assume that $b \in \bar{\Lie{t}}$, then  there exists a unique holomorphic vector field $\tilde{X}$ on $\PE$ such that $\pi_{*}\tilde{X}=X$.
As a result,
we are able to define the Hamiltonian  functions $l_{k}(b)$ on $\PE$ such that
$kd(l_{k}(b))=\iota_{\tilde{X}}\omega_{k}$. However, if $b$ does not belong to $\bar{\Lie{t}}$, then the corresponding holomorphic vector field does not lift to a holomorphic vector field on $\PE$.  Nevertheless, we are still able to define $l_k(b)$. In order to do that,  we use the following proposition proved in \cite{H4}.

%Let $G=Ham(M,\omega)$ be the group of Hamiltonian isometries of $(M,\omega)$ and $\Lie { g}$ be its Lie algebra. Let $\Lie{g}_{E} $ be the space of all Hamiltonian  holomorphic vector fields $X$ on $M$ that are liftable to a holomorphic vector field $\tilde{X}$ on $\PE$. Let $\bar{\Lie{g}}$ be the space of all Hamiltonian s whose gradient vector fields is in $\Lie{g}$ and $\bar{\Lie{g}}_{E}$ be the space of all Hamiltonian s whose gradient vector fields is in $\Lie{g}_{E}$ (including constant functions).

\begin{prop}\label{prop3}
For any holomorphic vector field $X$ on $M$, there exists a unique smooth $u_{X} \in \Gamma(\textup{End}(E))$ such that
\begin{align*}
&\Lambda_{\omega_{\infty}} \pa ( \bar{\partial}u_{X}-\iota_{X}F_{h})=0,\\
&
\int_{M} \textup{tr}(u_{X})\omega_{\infty}^m=0.
\end{align*}
Moreover, there exists a holomorphic vector field $\tilde{X}$ on $\PE$ such that $\pi_{*}\tilde{X}=X$ if and only if $\bar{\partial}u_{X}-\iota_{X}F_{h}=0$.

\end{prop}

If $f \in\bar{\Lie{g}}_{E}$ and $X$ be the gradient vector field corresponding to $b$, we can explicitly compute $l_{k}(b)$ in terms of $u_{X}$. Indeed, we have the following.

\begin{lem}\label{lem3}
 Suppose that holomorphic vector field $X$ has a holomorphic lift $\tilde{X}$ to $\PE$, then
 $\iota_{\tilde{X}}\omega_{g}=d \theta_{X}$, where $\theta_{X}=\textup{Tr}(u_{X} \lambda(h))$. Moreover,
 if $f \in \bar{\Lie{g}}_{E}$ such that $d f=\iota_{X}\omega_{\infty}$, then $d(\theta_{X}+kf)=\iota_{\tilde{X}}\omega_{k}$.

\end{lem}
Inspired by the proceeding lemma, we define the lift of elements of $\bar{\mathfrak g}$ to $\PE$.

\begin{definition}\label{def4}
We define \begin{align*}   &l_{k}: \bar{\Lie{g}} \to \mathcal  C^{\infty}(\PE)\\& f \in \bar{\Lie{g}} \mapsto l_{k}(f)=f+k^{-1}\theta_{X}, \end{align*} where $X$ is the holomorphic vector field on $M$ such that $\iota_{X}\omega=d f$ and $\theta_{X}=\textup{Tr}(u_{X}\lambda(h))$.
\end{definition}

 Suppose that $X_{s} \in \Lie{t}$.  Then there exists a holomorphic vector field $\tilde{X_{s}}$ on $\PE$ so that $\pi_{*} \tilde{X_{s}}=X_{s}$. Moreover from the definition of the function $ f \to l_k(f)$, we conclude that $d l_{k}( S(\omega_{\infty}))=k^{-1} \iota_{\tilde{X_{s}}}\omega_{k}$.

 Let $A$ be a vector space on which the group $T$ acts. Let $A^T$ be the subspace of $T$ invariant elements of $A$.
 The main goal of this section is to prove the following proposition.

\begin{prop}\label{prop4}
Let $h_{\textrm{HE}}$ be the Hermitian-Einstein metric on $E$ with respect to $\omega_{\infty}$, i.e.
$\Lambda_{\omega_{\infty}} F_{(E,h_{\textrm{HE}})}=\mu I_{E}$,
where
$\mu$ is the slope of the bundle $E$. Then there exist  $\eta_{0}, \eta_{1}, \dots \in \mathcal C^{\infty}(M)^T$ , $\Phi_{0},\Phi_1,\dots \in \Gamma(M,W)^T$,
 $\varphi_{0}, \varphi_{1}, \dots \in \mathcal C^{\infty}(\PE)^T$ and
 $b_{0},b_{1}, \dots \in \bar{\Lie{k}}$ such that for any positive integer
$p$, if $$\varphi_{k,p}= \sum_{j=2}^{p}\eta_{j}k^{-j+2}+ \sum_{j=2}^{p}\Phi_{j}k^{-j+1}+ \sum_{j=2}^{p}\varphi_{j}k^{-j},$$ and $$b_{k,p}= \sum_{j=0}^{p} k^{-j}b_{j},$$ then $$S(\omega_{k}+\ddbar \varphi_{k,p})+\frac 12\langle \nabla l_{k}(b_{k,p}), \nabla \varphi_{k,p}\rangle-l_{k}(b_{k,p})=O(k^{-p-1}).$$ Here the gradient and inner product are computed with respect to the K\"ahler metrics $\omega_{k}$. Moreover $b_{0}=r(r-1)$ and $b_{1}=S(\omega_{\infty})$.

 \end{prop}

Define $A_{1}(\omega, h)=S(\omega)I_{E}+\frac{i}{2\pi}\Lambda_{\omega} F_{h}^{0} $ and $S_{1}(\omega, h)= \textup{Tr}(A_{1}(h,\omega)\lambda(h))$, where $F_h^0$ is the traceless part of $F_h$.

\begin{prop}\label{prop5}
Suppose that  $\omega_{\infty} \in 2\pi c_{1}(L)$ is an extremal K\"ahler
metric on $M$ and $h_{\textrm{HE}}$ is a
Hermitian-Einstein  metric on $E$ with  the $\omega_{\infty}$-slope $\mu$. Then we have
\begin{align*}&A_{1,1}:=\frac{d}{dt}\Big|_{t=0}A_{1}(\omega_{\infty}+it\overline{\partial}\partial
\eta, h_{\textrm{HE}}(I+t\phi))\\&=\big(\mathcal{D}^*\mathcal{D}\eta-  \frac{1}{2}  \langle  \nabla_{\omega_{\infty}}S(\omega_{\infty}),  \nabla_{\omega_{\infty}} \eta \rangle_{\omega_{\infty}}\big)I_{E} \\&+
\frac{i}{2\pi}\Big\{(\Lambda_{\omega_{\infty}}\overline{\partial}\pa
\Phi+2\Lambda^2_{\omega_{\infty}}(F_{h_{\textrm{HE}}}\wedge
(i\overline{\partial}\partial \eta)
))\Big\}^{0},\end{align*}
where $\mathcal{D}^*\mathcal{D}$ is Lichnerowicz operator (cf.
\cite{D3}*{Page 515}) and $\{\Sigma\}^{0}$ is the traceless part of $\Sigma$, i.e. $\{\Sigma\}^0=\Sigma- \frac{1}{r}tr(\Sigma)$. Note that we use the operator $\pa$ to denote the covariant derivative of sections of the bundle $\textup{End}(E)$.
\end{prop}

\begin{proof}
Define
$f(t)=\Lambda_{\omega_{\infty}+it\overline{\partial}\partial
\eta}F_{(h_{\textrm{HE}}(I+t\phi))}$. Then we have
$$mF_{(h_{\textrm{HE}}(I+t\phi))} \wedge
(\omega_{\infty}+it\overline{\partial}\partial \eta)^{m-1}= f(t)
(\omega_{\infty}+it\overline{\partial}\partial \eta)^{m}.$$
Differentiating with respect to $t$ at $t=0$, we obtain
$$m \overline{\partial}\pa
\phi \wedge \omega_{\infty}^{m-1}+m(m-1)F_{h_{\textrm{HE}}}\wedge
(i\overline{\partial}\partial \eta)\wedge
\omega_{\infty}^{m-2}=f'(0)\omega_{\infty}^m+mf(0)(i\overline{\partial}\partial
\eta)\wedge \omega_{\infty}^{m-1}.$$ Since $f(0)=\mu I_{E}$, we
get
$f'(0)=\Lambda_{\omega_{\infty}}\overline{\partial}\pa \phi
+2\Lambda^2_{\omega_{\infty}}(F_{h_{\textrm{HE}}}\wedge
(i\overline{\partial}\partial \eta))-\mu
\Lambda_{\omega_{\infty}}(i\overline{\partial}\partial \eta)I_{E}
$.
On the other hand (cf. \cite{D3}*{pp. 515, 516}.)
$$\frac{d}{dt}\Big|_{t=0}S(\omega_{\infty}+it\overline{\partial}\partial
\eta)=\mathcal{D}^*\mathcal{D}\eta-\frac{1}{2} \langle  \nabla_{\omega_{\infty}}S(\omega_{\infty}),  \nabla_{\omega_{\infty}} \eta \rangle_{\omega_{\infty}}.$$
The proposition follows from the above two equations.
\end{proof}

\begin{lem} \label{lem4}

Suppose that  $\omega_{\infty} \in 2\pi c_{1}(L)$ is an extremal metric on $M$ and $h_{\textrm{HE}}$ be a
Hermitian-Einstein  metric on $E$, i.e.
$\Lambda_{\omega_{\infty}} F_{(E,h_{\textrm{HE}})}=\mu I_{E}$, where
$\mu$ is the $\omega_{\infty}$-slope of the bundle $E$. We have
\begin{align*}S_{1,1}&:=\frac{d}{dt}\Big|_{t=0}S_{1}(\omega_{\infty}+it\overline{\partial}\partial
\eta,h_{\textrm{HE}}(I+t\phi))\\&= \mathcal{D}^*\mathcal{D}\eta- \frac{1}{2}\langle  \nabla_{\omega_{\infty}}S(\omega_{\infty}),  \nabla_{\omega_{\infty}} \eta \rangle_{\omega_{\infty}}\\&+\frac{i}{2\pi}\textup{Tr}\Big( \big\{
\Lambda_{\omega_{\infty}}\overline{\partial}D
\phi+2\Lambda^2_{\omega_{\infty}}(F_{h_{\textrm{HE}}}\wedge
(i\overline{\partial}\partial \eta))\big\}^{0} \lambda(h_{\textrm{HE}}) \Big).\end{align*}

\end{lem}
\begin{proof}

The proof follows from the previous proposition and the fact that $$\{ \Lambda_{\omega_{\infty}}F_{(E,h_{\textrm{HE}})}\}^{0}=0.$$ Note that
\begin{align*}\frac{d}{dt}\Big|_{t=0}& \textup{Tr} \Big(  \{\Lambda_{\omega_{\infty}+it\overline{\partial}\partial
\eta}F_{(h_{\textrm{HE}}(I+t\phi))}\}^{0}\lambda(h_{\textrm{HE}}(I+t\phi))  \Big)\\&= \textup{Tr} \Big( \frac{d}{dt}\Big|_{t=0} \{\Lambda_{\omega_{\infty}+it\overline{\partial}\partial\eta}F_{(h_{\textrm{HE}}(I+t\phi))}\}^{0}\lambda(h_{\textrm{HE}})  \Big)
\\&+ \textup{Tr} \Big(  \{\Lambda_{\omega_{\infty}}F_{h_{\textrm{HE}}}\}^{0}\frac{d}{dt}\Big|_{t=0} \lambda(h_{\textrm{HE}}(I+t\phi))  \Big)\\&=
\textup{Tr}\Big(A_{1,1}(\eta, \phi)\lambda(h_{\textrm{HE}})  \Big).
\end{align*}

\end{proof}

Since $T$ is a compact group, by the uniqueness of the Hermitian-Einstein metric, $h$ is invariant under $T$.

\begin{lem}\label{lem5}
Suppose that $E$ is Mumford stable and $h$ is a Hermitian-Einstein metric with respect to $\omega_{\infty}$, i.e. $\Lambda_{\omega_{\infty}}F_{h}=\mu I_{E}$.
Then $h$ is invariant under the action of $T$.

\end{lem}

\begin{cor}\label{cor3}

The scalar curvature of $\omega_{k} $ is invariant under the action of $T$.

\end{cor}

The above two results follows from the uniqueness of the Hermitian-Einstein metric.

\begin{cor}\label{cor2}
The map
\begin{align*}
&\mathcal  C^{\infty}(M)\oplus \Gamma(M,W)\oplus \bar{\Lie{g} }\to \mathcal C_{0}^{\infty}(M)\oplus \Gamma(M,W)\\
&(\eta, \Phi, b)  \mapsto S_{1,1}(\eta, \Phi)+\frac 12\langle  \nabla_{\omega_{\infty}}S_{\infty},  \nabla_{\omega_{\infty}} \eta \rangle_{\omega_{\infty}}-b
\end{align*}
 is surjective. Here $S_{\infty}=S$ is the scalar curvature of $\omega_{\infty}$ and $\mathcal  C_{0}^{\infty}(M)$ is the space of smooth functions  $\eta$ on $M$ such that
 \[
 \int_{M}\eta \omega_{\infty}^m=0.
 \]
  Moreover, the equivariant version is also valid, that is,
   \begin{align*}
     &\mathcal C^{\infty}(M)^T\oplus \Gamma(M,W)^T\oplus \bar{\Lie{k} }\to \mathcal   C_{0}^{\infty}(M)^T\oplus \Gamma(M,W)^T\\
   &(\eta, \Phi, b)  \mapsto S_{1,1}(\eta, \Phi)+\frac 12\langle  \nabla_{\omega_{\infty}}S_{\infty},  \nabla_{\omega_{\infty}} \eta \rangle_{\omega_{\infty}}-b
   \end{align*} is surjective.

\end{cor}

\begin{lem}\label{lem6}

Let $\eta \in \mathcal C^{\infty}(M)$ and $\varphi \in\mathcal  C^{\infty}(\PE)$. Then $$\langle  \nabla_{\omega_{k}}\varphi,  \nabla_{\omega_{k}} \eta \rangle_{\omega_{k}}=O(k^{-1}).$$ Moreover if $\varphi \in  \mathcal C^{\infty}(M)$, then
$$\langle  \nabla_{\omega_{k}}\varphi,  \nabla_{\omega_{k}} \eta \rangle_{\omega_{k}}=k^{-1}\langle  \nabla_{\omega_{\infty}}\varphi,  \nabla_{\omega_{\infty}} \eta \rangle_{\omega_{\infty}}+O(k^{-2}).$$
\end{lem}

Before we give the proof of proposition, we explain how to find $\varphi_{k,2}$ and $b_{k,2}$.
%\begin{lem}%Let $h_{\textrm{HE}}$ be the Hermitian-Einstein metric on $E$ with respect to $\omega_{\infty}$, i.e.
%$\Lambda_{\omega_{\infty}} F_{(E,h_{\textrm{HE}})}=\mu I_{E}$, where$\mu$ is the slope of the bundle $E$. There exist  $\eta_{2}\in C^{\infty}(M)$, $\Phi_{2}\in \Gamma(M,W)$,
% $\varphi_{2} C^{\infty}(\PE)$ and $b_{2} \in \bar{\Lie{g}} $ such that $$S(\omega_{k})-l(r(r-1)+k^{-1}S(\omega))=O(k^{-2}),$$$$S(\omega_{k}+\ddbar \varphi_{k,2})+\nabla l(b_{k,2}). \nabla (  \varphi_{k,2})-l(b_{k,2})=O(k^{-3}).$$ Here $\varphi_{k,2}=\eta_{2}+ k^{-1}  \Phi_{2}+k^{-2} \varphi_{2}$ and $b_{k,2}=r(r-1)+k^{-1}S(\omega)+k^{-2}b_{2}$.
%\end{lem} %\begin{proof}
We can write $$S(\omega_{k})=r(r-1)+S_{1}k^{-1}+S_{2}k^{-2}+\dots.$$ Note that Corollary \ref{cor3} implies that $S(\omega_{k})$ is invariant under the action of $T$. Thus, $S_{i} \in \mathcal C^\infty(\PE)^T$,
where $S_{1}=S_{1}(\omega_{\infty},h_{\textrm{HE}})$. For any smooth function $\varphi $ on $\PE$, we have $$S(\omega_{k}+k^{-1}\ddbar \varphi)= r(r-1)+(S_{1}+\triangle_{V}(\triangle_{V}-r)\varphi)k^{-1}+O(k^{-2}),$$
\begin{align*}S(\omega_{k}+k^{-2}\ddbar \varphi)
= r(r-1)+S_{1}k^{-1}+(S_{2}+\triangle_{V}(\triangle_{V}-r)\varphi)k^{-2}+O(k^{-3}).
\end{align*} Hence for $\eta \in \mathcal C^\infty(M)$, $\Phi \in \Gamma(M,E)$ and $\varphi \in \mathcal C^\infty(\PE)$, we have
\begin{align*}
&S(\omega_{k}+\ddbar \eta +k^{-1}\ddbar \Phi+ k^{-2}\ddbar \varphi)\\
&= r(r-1)+S_{1}k^{-1}+(S_{2}+S_{1,1}(\eta, \Phi)+\triangle_{V}(\triangle_{V}-r)\varphi)k^{-2}+O(k^{-3}).
\end{align*}
Therefore \begin{align*}S(&\omega_{k}+\ddbar \eta+ k^{-1}\ddbar \Phi+k^{-2}\ddbar \varphi)-l(r(r-1)+k^{-1}S(\omega)+k^{-2}b_{2})\\&=k^{-2}\Big( S_{2}+\triangle_{V}(\triangle_{V}-r)\varphi+S_{1,1}(\eta, \Phi)-b_{2}-\Theta_{s}\Big)+O(k^{-3})\end{align*}
for some smooth function $\Theta_s$.
On the other hand
\begin{align*}
&\langle \nabla l(r(r-1)+k^{-1}S(\omega_{\infty})+k^{-2}b_{2}), \nabla ( \eta+ k^{-1} \Phi+k^{-2} \varphi_{2})\rangle\\
&=k^{-2}\langle  \nabla_{\omega_{\infty}}S(\omega_{\infty}),  \nabla_{\omega_{\infty}} \eta \rangle_{\omega_{\infty}}+O(k^{-3}).
\end{align*} Now we can find $\varphi_{2} \in \mathcal C^\infty(\PE)^T$ such that $$\triangle_{V}(\triangle_{V}-r)\varphi_{2}-b_{2}-\Theta_{s} \in \mathcal C^\infty(M)\oplus\Gamma(M,W).$$ Lemma \ref{lem6} implies that $\Theta_{s}$ is invariant under the action of $T$. Applying Lemma \ref{lem4} implies that
%On the other hand $S_{1,1}(\eta, \Phi)-\langle  \nabla_{\omega_{\infty}}S(\omega_{\infty}),  \nabla_{\omega_{\infty}} \eta \rangle_{\omega_{\infty}}$ is surjective.
there exist $\eta_{2} \in \mathcal C^\infty(M)^T$ and $\phi_{2} \in \Gamma(M,W)^T$ and $b_{2}\in $such that $$S_{1,1}(\eta_{2}, \Phi_{2}) +\frac{1}{2} \langle  \nabla_{\omega_{\infty}}S(\omega_{\infty}),  \nabla_{\omega_{\infty}} \eta_{2} \rangle_{\omega_{\infty}}=\triangle_{V}(\triangle_{V}-r)\varphi_{2}-b_{2}-\Theta_{s}.$$ Hence $$S(\omega_{k})-l(r(r-1)+k^{-1}S(\omega))=O(k^{-2}),$$$$S(\omega_{k}+\ddbar \varphi_{k,2})+\frac 12\langle\nabla l(b_{k,2}), \nabla (  \varphi_{k,2})\rangle-l(b_{k,2})=O(k^{-3}),$$ where $\varphi_{k,2}=\eta_{2}+ k^{-1}  \Phi_{2}+k^{-2} \varphi_{2}$ and $b_{k,2}=r(r-1)+k^{-1}S(\omega)+k^{-2}b_{2}$. Note that $\varphi_{k,2} \in \mathcal C^\infty(\PE)^T$, since $\eta_{2}, \phi_{2}$ and $\varphi_{2}$ are invariant under the action of $T$. 

% \begin{prop}
%Let $h_{\textrm{HE}}$ be the Hermitian-Einstein metric on $E$ with respect to $\omega$, i.e.$\Lambda_{\omega_{\infty}} F_{(E,h_{\textrm{HE}})}=\mu I_{E}$, where $\mu$ is the slope of the bundle $E$. There exist smooth functions $\eta_{0}, \eta_{1}, \dots$ on $M$, $\Phi_{0},\dots \in \Gamma(M,W)$,
%smooth functions $\varphi_{0}, \varphi_{1}, \dots$ on $\PE$ and $b_{0},b_{1}, \dots \in \bar{\Lie{g}} $such that for any positive integer $p$, if $$\varphi_{k,p}= \sum_{j=2}^{p}\eta_{j}k^{-j+2}+ \sum_{j=2}^{p}\Phi_{j}k^{-j+1}+ \sum_{j=2}^{p}\varphi_{j}k^{-j},$$ and $$b_{k,p}= \sum_{j=0}^{p} k^{-j}b_{j},$$ then $$S(\omega_{k}+\ddbar \varphi_{k,p})+\nabla l_{k}(b_{k,p}). \nabla \varphi_{k,p}-l_{k}(b_{k,p})=O(k^{-p-1}).$$ Here the gradient and inner product are computed with respect to the K\"ahler metrics $\omega_{k}$. Moreover $b_{0}=r(r-1)$ and $b_{1}=S(\omega_{\infty})$.

%\end{prop}

\begin{proof}[Proof of Proposition \ref{prop4}]

We prove it by induction on $p$. Suppose that we have chosen $\eta_{2},\dots \eta_{p-1} \in \mathcal C^\infty(M)^T,$ $\Phi_{2},\dots \Phi_{p-1} \in \Gamma(M,W)^T$, $\varphi_{2},\dots \varphi_{p-1} \in \mathcal C^\infty(\PE)^T$ and $b_{0}, \dots b_{p-1} \in $ such that $$S(\omega_{k}+\ddbar \varphi_{k,p-1})+\langle\nabla l(b_{k,p-1}), \nabla \varphi_{k,p-1}\rangle-l(b_{k,p-1})=k^{-p}\epsilon_{p}+O(k^{-p-1}).$$ We have

\begin{align*}S(&\omega_{k,p-1}+k^{-p+2}\ddbar\eta_{p}+ k^{-p+1} \ddbar \Phi_{p}+k^{-p}\ddbar \varphi_{p} )\\&=S(\omega_{k,p-1})+k^{-p}(\triangle_{V}(\triangle_{V}-r)\varphi_{p}+S_{1,1}(\eta_{p}, \Phi_{p}))+O(k^{-p-1}).\end{align*}
On the other hand,
\begin{align*}&\langle\nabla  l(b_{k,p-1}+k^{-p}b_{p}), \nabla (\varphi_{k,p-1}+k^{-p+2}\eta_{p}+ k^{-p+1}  \Phi_{p}+k^{-p} \varphi_{p})\rangle-l(b_{k,p-1}+k^{-p}b_{p})\\&=\langle\nabla l(b_{k,p-1}), \nabla \varphi_{k,p-1}\rangle-l(b_{k,p-1})+k^{-p}(\langle  \nabla_{\omega_{\infty}}S(\omega_{\infty}),  \nabla_{\omega_{\infty}} \eta \rangle_{\omega_{\infty}}-b_{p})+O(k^{-p-1}).\end{align*}

Corollary implies that we there exist $\eta_{p} \in \mathcal C^\infty(M)^T,\Phi_{p} \in \Gamma(M,W)^T, \varphi_{p} \in \mathcal C^\infty(\PE)^T$ and $b_{p} \in $ such that $$\triangle_{V}(\triangle_{V}-r)\varphi_{p}+S_{1,1}(\eta_{p}, \Phi_{p})+\frac{1}{2} \langle  \nabla_{\omega_{\infty}}S(\omega_{\infty}),  \nabla_{\omega_{\infty}} \eta \rangle_{\omega_{\infty}}-b_{p}-\epsilon_{p}=\textrm{Constant}.$$ This concludes the proof.

%%%%%%%%%%%%%%%%%%%%%%%%%

%We have $$S(\omega_{k})=r(r-1)+S_{1}k^{-1}+S_{2}k^{-2}+\dots.$$Here $S_{1}=S(\omega)+tr(F_{h}\lambda(h))$. For any smooth function $\varphi $ on $\PE$, we have $$S(\omega_{k}+k^{-1}\ddbar \varphi)= r(r-1)+(S_{1}+\triangle_{V}(\triangle_{V}-r)\varphi)k^{-1}+O(k^{-2}),$$$$S(\omega_{k}+k^{-2}\ddbar \varphi)= r(r-1)+S_{1}k^{-1}+(S_{2}+\triangle_{V}(\triangle_{V}-r)\varphi)k^{-2}+O(k^{-3}).$$ Hence for $\eta \in \mathcal C^\infty(M)$, $\Phi \in \Gamma(M,E)$ and $\varphi \in \mathcal C^\infty(\PE)$, we have $$S(\omega_{k}+\ddbar \eta +k^{-1}\ddbar \Phi+ k^{-2}\ddbar \varphi)= r(r-1)+S_{1}k^{-1}+(S_{2}+S_{1,1}(\eta, \Phi)+\triangle_{V}(\triangle_{V}-r)\varphi)k^{-2}+O(k^{-3}).$$ Now by choosing $f_{1}=S(\omega)$, .....????

%We have $$S(\omega_{k}+\sum_{j=2}^N\ddbar k^{-j+2}\eta_{j} +\sum_{j=2}^N\ddbar k^{-j+1}\ddbar \Phi_{j}+ \sum_{j=2}^N\ddbar k^{-j}\ddbar \varphi_{j})=\sum_{\alpha=0}^{\infty} k^{-\alpha}\tilde{s_{\alpha}},$$ where $$\tilde{S_{N}}=S_{N}+\triangle_{V}(\triangle_{V}-r)\varphi_{N}+S_{1,1}(\eta_{N},\Phi_{N})+LOT_{N},$$where $LOT_{N}$ involves only ????.

%%%%%%%%%%%%%%%%%%%%%%%%%

\end{proof}

\begin{definition}\label{def5}

Define $\omega_{k,p}=\omega_{k}+ \ddbar \varphi_{k,p}$. For any positive integer $p$ and any $b \in \bar{\Lie{g}}$, we define $l_{k,p}(b)=l_{k}(b)-\frac{1}{2}\langle\nabla l_{k}(b), \nabla \varphi_{k,p}\rangle_{\omega_{k,p}}$.

\end{definition}

The following lemma is straightforward.

\begin{lem}\label{lem7}
Let $b \in \bar{\Lie{g}}_{E}$ and  $X$ be the holomorphic vector fields on $M$ such that $db=\iota_{X}\omega$. Suppose that $\tilde{X}$ is the holomorphic lift of $X$ to $\PE$.
Then $k^{-1}d l_{k,p}= \iota_{\tilde{X}}\omega_{k,p}$.

\end{lem}

\begin{cor}\label{cor4}

We have $$S(\omega_{k,p})-l_{k,p}(b_{k,p})=O(k^{-p-1}).$$

\end{cor}

% For a fixed large integer $p$, we would like to solve the following equation:

%$$S(\omega_{k,p}+\ddbar \varphi)+\frac{1}{2}\nabla l_{p}(f_{k,p}+f). \nabla \varphi -l_{p}(f_{k,p}+f))=0,$$ where the gradient and inner product are computed with respect to the K\"ahler metrics $\omega_{k,p}$.  On the other hand, we have  $$S(\omega_{k,p}+\ddbar \varphi)=S(\omega_{k,p})+L_{k,p}(\varphi)+ Q_{k,p}(\varphi),$$ where $Q_{k,p}$ is the nonlinear part of the scalar curvature operator at $\omega_{k,p}$. Now we can rewrite the equation as follows:

%$$L_{k,p}(\varphi)+\frac{1}{2}\nabla l_{p}(f_{k,p}). \nabla \varphi -l_{p}(f)= l_{p}(f_{k,p})-S(\omega_{k,p})-Q_{k,p}(\varphi)-\frac{1}{2}\nabla l_{p}(f). \nabla \varphi. $$

%\begin{prop}

%For any positive integer $p$ and $k \gg 0$, the linear map \begin{align*}  &L^2_{s+4}(\PE) \times \bar{\Lie{g}} \to L^2_{s}(\PE)\\ (\varphi, f) &\to L_{k,p}(\varphi)+\frac{1}{2}\nabla l_{p}(f_{k,p}). \nabla \varphi -l_{p}(f) \end{align*} has a right inverse $G$ and $||G||_{op} \leq Ck^3$.

%\end{prop}

%Therefore it suffices to solve the following fixed point problem: $$G \big(l_{p}(f_{k,p})-S(\omega_{k,p})-Q_{k,p}(\varphi)-\frac{1}{2}\nabla l_{p}(f) \big)=(\varphi, f).$$

\section{Proof of Theorem \ref{thm2}}

The goal of this section is to prove Theorem \ref{thm2}.
%\begin{prop}%Suppose that $X_{s} \in \Lie{g}_{E}$. Then for any $k, p \gg 0$, we can find $\phi \in \mathcal C^\infty(\PE)$ and $b \in \bar{\Lie{g}}$ such that $$  S(\omega_{k,p}+ \ddbar \phi)+\frac{1}{2}\langle \nabla l(b), \nabla \phi \rangle=l(b).$$ Here $\nabla$ and $\langle , \rangle$ are taken with respect to $\omega_{k,p}$. Moreover $b $ has the following expansion.$$b=r(r-1)+k^{-1}S(\omega)-k^{-2}\pi_{N}(\Sigma_{E})+O(k^{-3}),$$ where $\pi_{N}:\mathcal C^\infty(M) \to \ker(\mathcal{D}^*\mathcal{D})$ and $$\Sigma_{E}=\frac{2}{r} \Lambda^2 (Ric(\omega)\wedge tr(iF_{h}))-\frac{2}{r(r+1)} \Lambda^2 (tr(iF_{h})\wedge tr(iF_{h}))+\frac{2}{r+1}\Lambda^2 tr(iF_{h}\wedge iF_{h})-\mu S(\omega)+\frac{tr(u_{X_{s}})}{r}.$$ See Proposition 3.1 for the definition of $u_{X_{s}}.$%$\end{prop}
We closely follow \cites{B, H3,Sz}. Before we give the proof, we go over some estimates from Hong and Br\"onnle.
Let's  fix a large positive integer $p$. In this section,  the operators $l=l_{k,p}$, $\mathcal{D}^*\mathcal{D}$ and  $\nabla$ and inner products are with respect to the metrics $\omega_{k,p}$.

%\begin{prop} (c.f. Hong, Br\"onlle)
%\begin{itemize}

%\item There exists a constant $C$ independent of $k$ such that $$C\Big( ||\mathcal{D}^*\mathcal{D}\phi||_{L^2(\omega_{k,p})}+||\phi||_{L^2(\omega_{k,p})}\Big) \geq ||\phi||_{L^2_{4}(\omega_{k,p})},$$ for all $\phi \in L^2$.

%\item There exists a constant $C$ independent of $k$ such that $$Ck^3 ||\mathcal{D}^*\mathcal{D}\phi||_{L^2(\omega_{k,p})} \geq ||\phi||_{L^2(\omega_{k,p})},$$ for all $\phi \in V_{k,p}^{\bot}$. Here $V_{k,p}=\ker (\mathcal{D}^*\mathcal{D})$.

%\end{itemize}
%\end{prop}

%\begin{lem}
%\begin{itemize}

%\item There exists a constant $C$ independent of $k$ such that $$||\tilde{X_{l(f)}}||_{\omega_{k,p}} \leq C k||f||,$$for all $f\in \Lie{g}$.

%\item There exists a constant $C$ independent of $k$ such that $$||l(f)||_{L^2(\omega_{k,p})} \leq C k||f||,$$for all $f\in \Lie{g}$.

%\end{itemize}

%\end{lem}

\begin{prop}[c.f. \cite{B}]\label{prop6}
Let $L_4^2=H^{4,2}$ be the Sobolev space of functions whose up to 4-th derivatives are in $L^2$ and $(L_4^2)^T$ is the subspace of $T$-invariant functions.
\begin{enumerate}
\item 
Let $p$ be a fixed positive integer. There exists a constant $C$ independent of $k$ such that the operators $$G_{k,p}: (L^2_4)^T \times \bar{\Lie{k}} \to (L^2)^T,$$ $$G_{k}(\phi, b)= \mathcal{D}^*\mathcal{D} \phi - \frac{1}{2} \langle \nabla S(\omega_{k,p}), \nabla \phi\rangle-\frac{1}{2}\langle \nabla l_{k,p}(b_{k,p})), \nabla \phi \rangle -l_{k,p}(b)$$ has right sided inverses $P_{k}$ satisfying $||P_{k}||_{op} \leq Ck^3.$ Note that $G_{k,p}$ is the linearization of the extremal operator at $(\omega_{k,p}, b_{k,p})$.

\item There exists a constant $C$ independent of $k$ such that
\begin{align*}
&||Q_{k,p}(\phi,b)-Q_{k,p}(\psi, b^{\prime})||_{L^2}\\
& \leq C \max(||(\phi,b)||_{L^2_{4}}, ||(\psi,b^{\prime})||_{L^2_{4}})||(\phi,b)-(\psi,b^{\prime})||_{L^2_{4}},
\end{align*}
where
\[
Q_{k,p}(\phi,b)=S(\omega_{k,p}+\ddbar \phi)+\frac{1}{2}\langle \nabla l_{k,p}(b_{k,p}+b), \nabla \phi\rangle-l_{k,p}(b_{k,p}+b)-G_{k,p}\]
 is the nonlinear part of the extremal operator at $(\omega_{k,p}, b_{k,p})$.

\end{enumerate}

\end{prop}

\begin{Rmk}

Our setting is slightly different from the setting in \cite{B}. In \cite{B}, Br\"onnle studied non-simple bundles over a base that does not admit nontrivial holomorphic vector field. However, the same proof as in  \cite{B} works in our setting.
\end{Rmk}
\begin{proof}[Proof of Theorem \ref{thm2} ]

%For a smooth function $\phi$, we have $$S(\omega_{k,p}+\ddbar \phi)= L_{k,p}\phi+Q(\omega_{k,p}, \phi),$$where $L_{k,p}$ is the linearization of the scalar curvature at $\omega_{k,p}$ and $Q$ is the nonlinear part of the scalar curvature.

We want to solve the following equation for $\phi \in \mathcal C^\infty(\PE)^T$ and $b \in \bar{\Lie{k}}$.

$$S(\omega_{k,p}+ \ddbar \phi)+\frac{1}{2}\langle \nabla l(b_{k,p}+b), \nabla \phi \rangle=l(b_{k,p}+b).$$
We can write it as the sum of linear and non-linear parts.
$$G_{k,p}(\phi,b)+Q_{k,p}(\phi,b)=0.$$ Then  in order to solve the equation, it suffices to solve the fixed point problem $$\mathcal{Q}_{k}(\phi, b)=(\phi, b),$$ where $\mathcal{Q}(\phi, b)=-P_{k}(Q_{k,p}(\phi,b))$.
We prove that the map $\mathcal{Q}$ is a contraction on the set $$B:=\{ (\phi, b)\in L^2_{4} \times \bar{\Lie{k}}\,\,\mid\, ||(\phi,b)||_{L^2_{4}}\leq 2C_{1}k^{-p+2} \}$$ for $p \geq 6$ and $k \gg 0$. First note that
\begin{align*}||\mathcal{Q}(0,0)||_{L^2}&=||P_{k}(Q_{k,p}(0,0))||_{L^2}\leq Ck^3||Q_{k,p}(0,0)||_{L^2}\\&=Ck^3||S(\omega_{k,p})-l_{k,p}(b_{k,p})||_{L^2}\leq C_{1}k^{-p+2}.\end{align*}
 Let $(\phi, b), (\phi^{\prime}, b^{\prime}) \in B$.  We have
 \begin{align*}||\mathcal{Q}(\phi, b)-\mathcal{Q}(\phi^{\prime}, b^{\prime})||_{L^2}& \leq ||P_{k}||_{op}||Q_{k,p}(\phi, b)-Q_{k,p}(\phi^{\prime}, b^{\prime})||_{L^2}\\& \leq Ck^3||Q_{k,p}(\phi, b)-Q_{k,p}(\phi^{\prime}, b^{\prime})||_{L^2}\\&\leq Ck^{5-p}||(\phi, b), (\phi^{\prime}- b^{\prime})||_{L_4^2}.\end{align*}
 Therefore, $$||\mathcal{Q}(\phi, b)-\mathcal{Q}(0,0)||_{L^2} \leq Ck^{5-p}||(\phi, b)||_{L^2}.$$ This implies that $$||\mathcal{Q}(\phi, b)||_{L^2}\leq ||\mathcal{Q}(0,0)||_{L^2}+Ck^{5-p}||(\phi, b)||_{L^2} \leq 2C_{1}k^{-p+2},$$ for $k\gg 0$ and $p \geq 6$. Hence $\mathcal{Q}(\phi, b): B\to B$ is a contraction for $k\gg 0$ and $p \geq 6$. Therefore, we can solve the equation for $\phi \in L^2_{4}$ and $b \in \bar{\Lie{k}}$. Now elliptic regularity implies that $\phi$ is smooth.

%$$S(\omega_{k,p}+k^{-p-1} \ddbar \phi)+\frac{k^{-p-1}}{2}\langle \nabla l(b_{k,p}+k^{-p-1}b), \nabla \phi \rangle=l(b_{k,p}+k^{-p-1}b).$$
%This is equivalent to the equation \begin{align*} &k^{-p-1}L_{k,p} \phi+ \frac{k^{-p-1}}{2}\langle \nabla l(b_{k,p})), \nabla \phi \rangle -k^{-p-1}l(b)\\&=  l(b_{k,p})-\frac{k^{-2p-2}}{2}\langle \nabla l(b), \nabla \phi \rangle -Q(\omega_{k,p}, k^{-p-1}\phi). \end{align*}

%Let $P_{k}$ be the right inverse of the operator  $$G_{k}(\phi, b)=L_{k,p} \phi+ \frac{1}{2}\langle \nabla l(b_{k,p}), \nabla \phi \rangle -k^{-2}l(b),$$ constructed in ?. Then in order to solve the equation, it suffices to solve the fixed point problem $$\mathcal{Q}_{k}(\phi, b)=(\phi, b),$$ where $$\mathcal{Q}_{k}(\phi, b)=P_{k}\Big(l(b_{k,p}-\frac{k^{-2p-2}}{2}\langle \nabla l(b), \nabla \phi \rangle -Q(\omega_{k,p}, k^{-p-1}\phi)\Big).$$We show that $\mathcal{Q}_{k}$ is a contraction map on a small ball.

%\begin{align*}||\mathcal{Q}_{k}(\phi,b)-\mathcal{Q}_{k}(\phi^{\prime}, b ^{\prime} )||&=||P_{k}\Big(\frac{k^{-2}}{2}\langle \nabla l(b), \nabla \phi \rangle +Q(\omega_{k,p}, \phi))-\frac{k^{-2}}{2}\langle \nabla l(b^{\prime}), \nabla \phi^{\prime} \rangle -Q(\omega_{k,p}, \phi^{\prime}\Big)||\\&\leq k^{-2}||P_{k}||_{op} ||\langle \nabla l(b), \nabla \phi \rangle -\langle \nabla l(b^{\prime}), \nabla \phi^{\prime} \rangle ||+||P_{k}||_{op}||Q(\omega_{k,p}, \phi)) -Q(\omega_{k,p}, \phi^{\prime}||.\end{align*}

\end{proof}

An immediate consequence of Theorem \ref{thm2} is the following.

\begin{Cor}
Let $(M,L)$ be a compact polarized manifold and $\omega_{\infty} \in c_{1}(L) $ be an extremal K\"ahler metric. Let  $X_{s}$ be the gradient vector field of the scalar curvature of $\omega_{\infty}$, i.e. $d S(\omega_{\infty})= \iota_{X_{s}}\omega_{\infty}$. Let $E$ be a Mumford stable holomorphic vector bundle over $M$. Suppose that all holomorphic vector fields on $M$ can be lifted to  holomorphic vector fields on $\mathbb{P}E^*$. Then there exist extremal metrics on  $(\mathbb{P}E^*,\mathcal{O}_{\mathbb{P}E^*}(1)\otimes  L^k)$ for $k \gg 0$.

\end{Cor}

\section{Hong's moment map setting and proof of Theorem \ref{thm1}}
 In this section, we follow \cite{H4} to prove Theorem \ref{thm1}. As before, let $(M,\omega_{\infty})$ be a K\"ahler manifold of dimension $m$ and $G$ be the group of Hamiltonian isometries of $(M, \omega_{\infty})$. Note that the Lie algebra of $G$ is the space of Hamiltonian vector fields on $(M,\omega_{\infty})$. Define $$\mathcal{N}=\{ f \in \mathcal C^\infty(M)\mid \iota_{X}\omega_{\infty}=df \textrm{    for some  } X \in \Lie{g }\}= \textup{Ker} ( \mathcal{D}^*\mathcal{D}). $$

%\begin{Def}

Any $\xi \in \Lie{g}$ defines a  holomorphic vector field $\xi^\#$ on $M$. For any $\xi \in \Lie{g}$, there exists a unique smooth function $f_{\xi} \in \mathcal C^\infty(M)$ such that \begin{equation}\label{eq6}\iota_{\xi^\#}\omega_{\infty}=df_{\xi}  \,\,\,\,\,\textrm{and}\,\,\,\,\, \int_{M}f_{\xi}\omega_{\infty}^m=0.\end{equation}

%\end{Def}
The following is a straightforward computation.

 \begin{prop}\label{prop7}

The map $\xi \in \Lie{g} \to f_{\xi}$ is an isomorphism of Lie algebras. Moreover, for any $g \in G$ and $\xi \in \Lie{g}$, we have $$f_{Ad(g)\xi}= f_{\xi}\circ \sigma_{g^{-1}},$$ where $\sigma_{g}: M \to M $ is defined by $\sigma_{g}(x)=g.x$.

 \end{prop}

 \begin{cor}\label{cor5}

 A function $f \in \mathcal{N}$ is in the center of $\mathcal{N}$ if and only if $f$ is $G$-invariant. Moreover, if $f \in \mathcal C^\infty(M)$ is a $G$-invariant function, then $\pi_{\mathcal{N}}(f)$ is in the center of $\mathcal{N}$, where $\pi_{\mathcal{N}}:\mathcal C^\infty(M) \to \mathcal{N}$ is the orthogonal projection.
 \end{cor}

 Let $Y$ be a K\"ahler manifold (open or compact without boundary). Suppose that the Lie algebra $\Lie{g}$ acts on $Y$. Then $[\xi_{1}^{\sharp}, \xi_{2}^{\sharp}]=[\xi_{1}, \xi_{2}]^{\sharp}$ for all $\xi_{1}, \xi_{2} \in \Lie{g}$. Integrating the action of $\Lie{g}$, we obtain an action of $G$ (an open neighborhood of identity in $G$) on $Y$. Therefore, there exists an equivariant moment map $\mu^Y:Y \to \Lie{g}$. Compose $\mu^Y$ with the map $\xi \in \Lie{g} \to f_{\xi}$ defined by \eqref{eq6}, we have an equivariant moment map $\mu^Y:Y \to \mathcal{N}.$ We apply this setting to the case when $Y$ is the smooth locus of moduli space of Hermitian-Einstein connections $\mathcal{M}$ on a smooth complex vector bundle $\mathcal{E}$ and the action of $G$ on $Y$. More precisely, let $\mathcal{E}$ be a smooth complex vector bundle of rank $r$ on $M$ and $h$ be a fixed hermitian metric on $\mathcal{E}$. We fix a holomorphic structure on $\det(\mathcal{E})$. Let $\mathcal{G}$ be the group of unitary gauge transformations of $(\mathcal{E}, h).$ Let $\mathcal{A}$ be the space of Hermitian-Einstien connections $A$ on $\mathcal{E}$ such that $(\mathcal{E}, \bar{\partial}_{A})$ is a simple holomorphic vector bundle and $A$ induces the fixed holomorphic structure on $\det(\mathcal{E})$ modulo the action of the unitary gauge group $\det(\mathcal{E})$.
 %\textbf{Elaborate on the definition of $\mathcal{A}$?}

 Now we define the moduli space of simple Hermitian-Einstein metrics on $\mathcal{E}$ as follows: $$\mathcal{M}=\frac{\mathcal{A}}{\mathcal{G}}.$$ One can compute the tangent space to $\mathcal{A}$ and the moduli space $\mathcal{M}$ (c.f. \cite{K}):
for any $A \in \mathcal{A}$, we have $$T_{A}\mathcal{A}=\{ \alpha \in \Omega^{1}(M, \textup{End}(\mathcal{E},h))\mid \nabla^{A}\alpha\in \Omega^{1,1}  \,\,\,\, \textrm{and}  \,\,\, \Lambda \nabla^{A}\alpha=0         \}.$$ Moreover if $[A] \in \mathcal{M}$ is a smooth point of $\mathcal{M},$ then $$T_{[A]}\mathcal{M}=\frac{\{ \alpha \in \Omega^{1}(M, \textup{End}(\mathcal{E},h))\mid  \nabla^{A}\alpha\in \Omega^{1,1}  \,\,\,\, \textrm{and}  \,\,\, \Lambda \nabla^{A}\alpha=0         \}}{\{\nabla_{A} s\mid  s \in \Gamma(M, \textup{End}(\mathcal{E},h))\}}.$$ Note that the moduli space $\mathcal{M}$ is not smooth in general. However, one can define an action of $\Lie{g}$ on $\mathcal{A}$ as follows:
Proposition \ref{prop3} tells that for any $A \in \mathcal{A}$ and any $X \in \Lie{g}$, there exists a unique $u_{X} \in \Gamma(M, \textup{End}(\mathcal{E}))$, depending on $A$, such that  $$\Lambda \partial_{A}(\bar{\partial}_{A}u_{X}-\iota_{X}F_{A})=0   \,\,\,\,\textrm{and} \,\,\,\,\, \int_{M} \textup{tr}(u	_{X})\omega^m=0. $$ For any $A \in \mathcal{A}$ and $X \in \Lie{g}$, define $$\theta_{X}(A)= -(-\partial_{A} g_{X}^*+\bar{\partial}_{A}g_{X}-\iota_{X}F_{A}) \in T_{A}\mathcal{A}. $$ Note that the vector field $\theta_{X}$ is the infinitesimal vector field on $\mathcal{A}$ induced by the action of $X$. Hong proved that the vector field $\theta_{X}$ can be descended to the moduli space $\mathcal{M}$. Moreover, he proved that $$[\theta_{X},\theta_{Y}]-\theta_{[X,Y]} \in d_{A}\Gamma(M,\textup{End}(\mathcal{E})).$$ This implies that on the moduli space $\mathcal{M}$, we have $[\theta_{X},\theta_{Y}]=\theta_{[X,Y]}$. Therefore we have an action of the Lie algebra $\Lie{g}$ on $\mathcal{M}$.

% \begin{prop}[\cite{H4}]

 %For any $A \in \mathcal{A}$ and any $X \in \Lie{g}$, there exists a unique $g_{X} \in \Gamma(M, End(\mathcal{E}))$, depending on $A$, such that  $$\Lambda \partial_{A}(\bar{\partial}_{A}g_{X}+\iota_{X}F_{A})=0   \,\,\,\,\textrm{and} \,\,\,\,\, \int_{M} tr(g_{X})\omega^m=0. $$  \end{prop}

% \begin{Def}

%For any $A \in \mathcal{A}$ and any $X \in \Lie{g}$, we define $$\theta_{X}(A)= -(-\partial_{A} g_{X}^*+\bar{\partial}_{A}g_{X}+\iota_{X}F_{A}) \in T_{A}\mathcal{A}. $$

 %\end{Def}

%\begin{prop}[\cite{H4}]

%$\theta_{X} \in T\mathcal{M}$. Moreover $$[\theta_{X},\theta_{Y}]-\theta_{[X,Y]} \in d_{A}\Gamma(M,End(\mathcal{E})).$$ \end{prop}

%\begin{cor}
%In the moduli space $\mathcal{M},$ we have $[\theta_{X},\theta_{Y}]=\theta_{[X,Y]}.$

%\end{cor}

\begin{prop}(\cite{H4})\label{prop8}
The map $\mu^e:\mathcal{M} \to \mathcal{N}$ given by $$\mu^{e}([A])= \pi_{\mathcal{N}}\big( \Lambda^2 \textup{tr}(iF_{A} \wedge iF_{A})\big)$$ is an equivariant moment map for the action of $G$ on $\mathcal{M}$.

\end{prop}

By the definition of $\mathcal{A}$, any connection $A \in \mathcal{A}$ induces the fixed holomorphic structure on $\det(\mathcal{E})$ up to the action of the gauge group. Therefore the function $$\frac{2}{r} \Lambda^2 (\textup{Ric}(\omega)\wedge \textup{tr}(iF_{A}))-\frac{2}{r(r+1)} \Lambda^2 (\textup{tr}(iF_{A})\wedge \textup{tr}(iF_{A}))-\mu S(\omega)+\frac{\textup{tr}(u_{X_{s}})}{r}$$ is independent of the choice of $A \in \mathcal{A}$.
We define the moment map $\mu:\mathcal{M} \to \mathcal{N}$ as follows:
\begin{align*}
&\qquad
\mu([A])=\mu^{e}([A])\\
&+\pi_{\mathcal{N}}\Big( \frac{2}{r} \Lambda^2 (\textup{Ric}(\omega)\wedge \textup{tr}(iF_{A}))-\frac{2}{r(r+1)} \Lambda^2 (\textup{tr}(iF_{A})\wedge \textup{tr}(iF_{A}))-\mu S(\omega)+\frac{\textup{tr}(u_{X_{s}})}{r}\Big). 
\end{align*}

\begin{lem}\label{lem8}

The moment map $\mu$ is equivarent if $c_{1}(L) =\lambda c_{1}(E)$ for some constant $\lambda \in \mathbb{Z}$.

\end{lem}

\begin{proof}
Since $c_{1}(L) =\lambda c_{1}(E),$ there exists a smooth function $\varphi$ on $M$ such that $\lambda \textup{tr}(iF_{A}))-\omega_{\infty}=\ddbar\varphi.$ Taking the trace with respect to $\omega_{\infty}$, we have $r\mu\lambda-m=\Delta \varphi,$ where $\mu$ is the slope of $E$. Using the Hermitian-Einstein condition, we obtain that $\varphi $ is constant and therefore $\lambda \textup{tr}(iF_{A}))=\omega_{\infty}.$ Therefore $$ \frac{1}{r} \Lambda^2 (\textup{Ric}(\omega_{\infty})\wedge \textup{tr}(iF_{A}))-\frac{1}{r(r+1)} \Lambda^2 (\textup{tr}(iF_{A})\wedge \textup{tr}(iF_{A}))-\mu S(\omega_{\infty})+\frac{\textup{tr}(u_{X_{s}})}{r}$$ is invariant under the action of $G$ since $\omega_{\infty}$ is invariant under the action of $G$.

\end{proof}
Following \cite{H4}, we define the following.
\begin{definition}\label{def6}

A holomorphic structure $A$ is called stable relative to the maximal torus $T$ if  there exists a connection $A_{\infty}$ in the orbit of $[A] \in \mathcal{M}$ such that $[A_{\infty}]$ is a smooth point of $\mathcal{M}$, $\mu([A_{\infty}])\in \Lie{t}$ and $$\frac{\partial \mu}{\partial A}([A_{\infty}]): T_{[A_{\infty}]} \to \frac{\bar{\Lie{k}} }{\bar{\Lie{t}} } $$ is surjective.

\end{definition}

\begin{rem}

%\begin{enumerate}
%\item The moduli space $\mathcal{M}$ is not smooth in general. But we expect to be able to work on the smooth part of it.
%\item

The moment map $\mu$ is not equivariant in general. If  $c_{1}(L) =\lambda c_{1}(E)$, then the notion of relative stability defined above is a GIT notion of stability introduced by Sz{\'e}kelyhidi (\cite{Sz1}). However, it is not clear how the notion of stability defined above is related to a GIT notion of stability in general .

%\end{enumerate}

\end{rem}

%\begin{thm}

%Suppose that $E$ is stable in the sense of the previous definition. Then $(\PE , \OPE \otimes L^k)$ admits extremal metrics for $k \gg 0.$

%\end{thm}

\begin{proof}[Proof of Theorem \ref{thm1}]

Theorem \ref{thm2} implies that for any $A \in \mathcal{A} $, we can find $\phi^{A} \in \mathcal C^\infty(\PE_{A})$ and $b^{A} \in \bar{\Lie{k}}$ such that $$  S(\omega_{k,p}^{A}+ \ddbar \phi^{A})+\frac{1}{2}\langle \nabla l(b^{A}), \nabla \phi^{A} \rangle=l(b^{A}).$$ Moreover $b^{A} $ has the following expansion.$$b^{A}=r(r-1)+k^{-1}S(\omega)-k^{-2}\mu(A)+k^{-3}R^{A}.$$ One can easily see through the computation that $\omega_{k,p}^{A}$, $\phi^{A}$ and $b^{A}$ depend smoothly on $A$.
Suppose that $A_{\infty}$ is in the orbit of $E$ and $\mu(A_{\infty}) \in \Lie{t}$. Define $\Phi(A,t)=\mu(A)+tR^{A}.$
Then $\Phi(A_{\infty}, 0)\in \bar{\Lie{t}}$ and $p_{1}(\frac{\partial \Phi}{\partial A}(A_{\infty}, 0))$ is surjective, where $p_{1}:\bar{\Lie{k}} \to \frac{\bar{\Lie{k}} }{\bar{\Lie{t}} }$ is the quotient map.  Therefore applying the  implicit function theorem, we find $A_{t}$ for small $t$ such that $\mu(A_{t})+tR^{A}\in \bar{\Lie{t}}$ for $t$ small enough; hence for $k=t^{-1} \gg 0$, we have $b^{A_{t}}=r(r-1)+tS(\omega)-t^{2}\mu(A_{t})+t^{3}R^{A_{t}}=r(r-1)+tS(\omega) \in \bar{\Lie{t}}.$ This implies that $\omega_{k,p}^{A_{k^{-1}}}+ \ddbar \phi^{A_{k^{-1}}}$ are extremal metrics for $k \gg0$. Note that $A_{k^{-1}}$ are compatible with the holomorphic structure of $E$ since they are all in the orbit of $A_{\infty}.$

\end{proof}

%\begin{rem}
 %Note that the same argument implies that  $(\PE , \OPE \otimes L^k)$ admits extremal metrics for $k \gg 0$ if $E$ is semi-stable in the sense of the Definition \ref{def6}.

 %\end{rem}

An holomorphic vector bundle $\mathcal E$ is called projectively flat, if the curvature of $E$ is of the form $c\cdot\textup{Id}\otimes\omega$, where $c$ is a constant and $\omega$ is the \ka metric of the base manifold $(M,L)$. Assume that $(M,L)$ is an extremal \ka manifold. Then by our result, for $k\gg 0$, there are extremal \ka metrics in the classes $\mathcal{O}_{\mathbb{P}E^*(1)}\otimes \pi^*(L^k)$ on $\PE$.

% ----------------------------------------------------------------

\end{document}